\documentclass[a4paper,10pt]{amsart}

\usepackage{graphicx}
\usepackage{geometry}
\usepackage{amsthm}
\usepackage{amssymb}
\usepackage{amsmath}
\usepackage{accents}
\usepackage{hyperref}
\usepackage{xcolor}
\usepackage{enumerate}
\usepackage{listings}
\usepackage{complexity}
\usepackage{blkarray}
\usepackage{braket}
\usepackage{lmodern}
\usepackage[english]{babel}

\usepackage[font=small,labelfont=bf]{caption}
\usepackage[font=small,labelfont=normalfont,labelformat=simple]{subcaption}
\graphicspath{{figs/}}

\newtheorem{theorem}{Theorem}

\newtheorem{problem}{Problem}

\author
{
Raphael Steiner 
}
\thanks{Department of Computer Science, Institute of Theoretical Computer Science, ETH Z\"{u}rich, Switzerland,  \texttt{raphaelmario.steiner@inf.ethz.ch}. The author was supported by an
ETH Zurich Postdoctoral Fellowship.}

\date{\today}

\title{Subdigraphs of prescribed size and outdegree}

\begin{document}
\maketitle

\begin{abstract}
In 2006, Noga Alon~\cite{alon} raised the following open problem: Does there exist an absolute constant $c>0$ such that every $2n$-vertex digraph with minimum out-degree at least $s$ contains an $n$-vertex subdigraph with minimum out-degree at least $\frac{s}{2}-c$~?

In this note, we answer this natural question in the negative, by showing that for arbitrarily large values of $n$ there exists a $2n$-vertex tournament with minimum out-degree $s=n-1$, in which every $n$-vertex subdigraph contains a vertex of out-degree at most $\frac{s}{2}-\left(\frac{1}{2}+o(1)\right)\log_3(s)$. 
\end{abstract}

\section{Introduction}
The following open problem was raised by Alon in 2006 in the article~\cite{alon} on digraph splitting. 
\begin{problem}[cf.~\cite{alon}, Problem 4.1]
    Does there exist an absolute constant $c>0$ such that the following holds for all pairs of natural numbers $(n,s)$?

\medskip

    Every $2n$-vertex digraph $D$ of minimum out-degree at least $s$ contains an $n$-vertex subdigraph of minimum out-degree at least $\frac{s}{2}-c$. 
\end{problem}

Using an elegant probabilistic argument, Alon proved in~\cite{alon} that a weaker bound holds, as follows: Every $2n$-vertex digraph of minimum out-degree at least $s$ contains an $n$-vertex subdigraph of minimum out-degree at least $\frac{s}{2}-O(\sqrt{s}\sqrt{\log s})$. 

In this short note, we resolve Problem~1 in the negative, by proving the following result via an explicit construction.

\begin{theorem}\label{statement}
    For an infinite sequence of numbers $n\in \mathbb{N}$, there exists a $2n$-vertex tournament whose minimum out-degree is $s=n-1$, while the minimum out-degree of every $n$-vertex subdigraph is at most $\frac{s}{2}-(\frac{1}{2}+o(1))\log_3(s)$. 
\end{theorem}

Following the notation from~\cite{alon}, for every $s\in \mathbb{N}$ denote by $d(s)$ the largest integer such that every $2n$-vertex digraph of minimum out-degree at least $s$ contains an $n$-vertex subdigraph of minimum out-degree at least $d(s)$. Alon's result in~\cite{alon} and Theorem~\ref{statement} yield $$\Omega(\log s) < \frac{s}{2}-d(s) < O(\sqrt{s}\sqrt{\log s}).$$ It remains an interesting open problem to close the gap between the lower and upper bounds. 

\paragraph{\textbf{Terminology}}
All digraphs considered in this paper have no loops or parallel arcs. For a digraph $D$ we denote by $V(D)$ its vertex set and by $A(D)\subseteq \{(u,v)\in V(D)^2|u \neq v\}$ its arc set. For a subset $X \subseteq V(D)$ of vertices we denote by $D[X]$ the subdigraph of $D$ induced by $X$, i.e.,  consisting of the vertex set $X$ and all arcs of $D$ going between vertices of $X$. By $\delta^+(D)$, we denote the \emph{minimum out-degree} of $D$, which is the smallest out-degree occurring in $D$ if $D$ is non-empty, and defined as $0$ when $D$ is the empty digraph.

\section{Proof of Theorem~\ref{statement}}
Theorem~\ref{statement} is an immediate consequence of the following result.

\begin{theorem}\label{thm:main}
For every integer $k \ge 0$ there exists a tournament $T_k$ on $3^k$ vertices in which every vertex has out-degree $\frac{3^k-1}{2}$, and such that the following holds: For every set $X \subseteq V(T_k)$ of size $|X|\le \frac{3^k-1}{2}$ we have $\delta^+(T_k[X]) \le \frac{\frac{3^k-1}{2}-k}{2}$.  
\end{theorem}

To deduce Theorem~\ref{thm:main} from Theorem~\ref{statement}, pick $k \ge 0$ arbitrarily, let $n:=\frac{3^k-1}{2}$, and let $D_k$ be the $2n$-vertex tournament obtained from $T_k$ by deleting an (arbitrarily selected) vertex. Since deleting a vertex can lower the out-degree by at most $1$, we have that $D_k$ is of minimum out-degree $s:=n-1=\frac{3^{k}-1}{2}-1$. Furthermore, for every subdigraph $D' \subseteq D_k$ on $n=\frac{3^k-1}{2}$ vertices, denoting by $X$ the set of vertices in $D'$, we have that $\delta^+(D') \le \delta^+(D_k[X])=\delta^+(T_{k}[X])\le \frac{\frac{3^k-1}{2}-k}{2}=\frac{s}{2}-(\frac{1}{2}+o(1))\log_3(s)$. Thus, the statement of Theorem~\ref{statement} follows from Theorem~\ref{thm:main}. 

\medskip

Let us now go about proving Theorem~\ref{thm:main}. 
The tournaments $(T_k)_{k \ge 0}$ have the following simple recursive definition: 
\begin{itemize}
    \item $T_0$ is the one-vertex-tournament. 
    \item For every integer $k \ge 0$, the tournament $T_{k+1}$ is obtained from the disjoint union of $3$ isomorphic copies of $T_k$ on vertex sets $A_k, B_k, C_k$ by adding all possible arcs from $A_k$ to $B_k$, from $B_k$ to $C_k$ and from $C_k$ to $A_k$.
\end{itemize}

It is immediate from this definition that for every $k \ge 0$ the digraph $T_k$ is a tournament on $3^k$ vertices and that every vertex has the same in- and out-degree, hence, the out-degree of each vertex equals $\frac{3^k-1}{2}$. We now prove the statement of Theorem~\ref{thm:main} by induction on $k$. 

\begin{proof}[Proof of Theorem~\ref{thm:main}]
For $k=0$, every induced subdigraph of $T_1$ has minimum out-degree $0$ by definition, which is equal to $\frac{\frac{3^0-1}{2}-0}{2}$, and thus the induction basis holds. For the induction step, assume that for some integer $k \ge 0$ we have established that $\delta^+(T_k[X]) \le \frac{\frac{3^k-1}{2}-k}{2}$ for every $X \subseteq V(T_k)$ of size at most $\frac{3^k-1}{2}$, and let us show the corresponding statement for $T_{k+1}$. 

So let $X \subseteq V(T_{k+1})=A_k \cup B_k \cup C_k$ be given arbitrarily such that $|X| \le \frac{3^{k+1}-1}{2}$. W.l.o.g assume $X \neq \emptyset$.  Denote $X_A:=X \cap A_k, X_B:=X \cap B_k, X_C:=X \cap C_k$, and let $x_A:=|X_A|, x_B:=|X_B|, x_C:=|X_C|$, so that $x_A+x_B+x_C=|X| \le \frac{3^{k+1}-1}{2}$. 

Consider first the case that $\min\{x_A,x_B,x_C\}=0$. Then possibly after relabeling, we may assume w.l.o.g. that $X_B=\emptyset$ and $X_A \neq \emptyset$. Pick a vertex $x \in X_A$. Since all its out-neighbors within $X$ are also contained in $A_k$, and since the copy of $T_k$ induced on $A_k$ is $\frac{3^k-1}{2}$-out-regular, this certifies that $\delta^+(T_{k+1}[X])\le \frac{3^{k}-1}{2}\le \frac{\frac{3^{k+1}-1}{2}-(k+1)}{2}$, as desired. 

Next, suppose that $x_A, x_B, x_C>0$. We then have, by construction of $T_{k+1}$, that
$$\delta^+(T_{k+1}[X])=\min\{\delta^+(T_{k+1}[X_A])+x_B,\delta^+(T_{k+1}[X_B])+x_C,\delta^+(T_{k+1}[X_C])+x_A \}.$$ 
Among the three numbers $x_A, x_B$ and $x_C$, at least two are of size at least $\frac{3^k+1}{2}$, or at least two are of size at most $\frac{3^{k}-1}{2}$. Without loss of generality (possibly after relabeling), we may assume that $x_A$ and $x_B$ have this property. We now proceed according to the two possible cases. 

\medskip

\paragraph{\textbf{Case 1.} $x_A, x_B \ge \frac{3^k+1}{2}$.} We want to estimate as follows:
$$\delta^+(T_{k+1}[X]) \le \delta^+(T_{k+1}[X_B])+x_C.$$
In order to estimate the minimum out-degree of $T_{k+1}[X_B]$, pick a proper subset $S\subseteq X_B$ of size exactly $\frac{3^k-1}{2}$. By inductive assumption (applicable since $T_{k+1}[B_k]$ is isomorphic to $T_k$), we have $\delta^+(T_{k+1}[S]) \le \frac{\frac{3^k-1}{2}-k}{2}$. Since adding a single vertex can increase the minimum out-degree by at most $1$, we conclude that $\delta^+(T_{k+1}[X_B]) \le \frac{\frac{3^k-1}{2}-k}{2}+\left(x_B-\frac{3^k-1}{2}\right)$. Putting things together, we obtain: 
$$\delta^+(T_{k+1}[X]) \le \frac{\frac{3^k-1}{2}-k}{2}+\left(x_B-\frac{3^k-1}{2}\right)+x_C$$
$$=\underbrace{(x_A+x_B+x_C)}_{\le \frac{3^{k+1}-1}{2}}-\underbrace{x_A}_{\ge \frac{3^k+1}{2}}-\frac{\frac{3^k-1}{2}+k}{2}$$
$$\le \frac{3^{k+1}-1}{2}-\frac{3^k+1}{2}-\frac{\frac{3^k-1}{2}+k}{2}$$
$$=\frac{\frac{3^{k+1}-1}{2}-(k+1)}{2},$$ as desired, and this concludes the proof in the first case.

\medskip

\paragraph{\textbf{Case 2.} $x_A, x_B \le \frac{3^k-1}{2}$.} In particular, $|X_A|\le \frac{3^k-1}{2}$, so that we can apply the inductive assumption to the set $X_A$ in the tournament $T_{k+1}[A_k]$, which is isomorphic to $T_k$. We conclude that $\delta^+(T_{k+1}[X_A])\le \frac{\frac{3^k-1}{2}-k}{2}$. We can now estimate the minimum out-degree in $X$ via
$$\delta^+(T_{k+1}[X]) \le \delta^+(T_{k+1}[X_A])+x_B$$
$$\le \frac{\frac{3^k-1}{2}-k}{2}+\frac{3^k-1}{2}$$
$$=\frac{\frac{3^{k+1}-1}{2}-(k+1)}{2},$$ thus proving the desired bound on the minimum degree also in the second case. This concludes the proof of the theorem. 
\end{proof}

\end{document}